\documentclass[12pt,reqno]{amsart}
\usepackage[margin=1in]{geometry}
\usepackage{amssymb,amsfonts,amsmath,amsthm}
\usepackage{mathrsfs}
\usepackage{enumerate}
\usepackage{tikz}
\usepackage{tikz-3dplot}
\usepackage{caption}
\usepackage{graphicx}
\usepackage{url}
\usetikzlibrary{calc,backgrounds}

\makeatletter
\let\@afterindentfalse\@afterindenttrue
\@afterindentfalse
\makeatother

\newtheorem{theorem}{Theorem}[section]
\newtheorem{lemma}[theorem]{Lemma}
\newtheorem{proposition}[theorem]{Proposition}
\newtheorem{corollary}[theorem]{Corollary}
\theoremstyle{definition}
\newtheorem{definition}{Definition}
\newtheorem{remark}{Remark}
\newtheorem{example}[theorem]{Example}

\numberwithin{equation}{section}

\def\C{\mathcal{C}}
\def\K{\mathbb{K}}

\newcommand{\pd}{\operatorname{pd}}
\newcommand{\reg}{\operatorname{reg}}

\begin{document}

\title[Complements of closed powers of cycles]
{Homological invariants of Edge Ideals associated to   powers of cycles}

	\author{Shahnawaz Ahmad Rather}
\address{Department of Mathematics \\
	Government Degree College  \\
	Baramulla \\
	Jammu and Kashmir 193101, India}
\email{nawaaz315@gmail.com}

	\author{S. Pirzada}
\address{Department of Mathematics \\
	University of Kashmir, Hazratbal, \\
	Jammu and Kashmir 190006, India}
\email{pirzadasd@kashmiruniversity.net}
	\author{M. Aijaz}
\address{Department of Mathematics \\
	Government Degree College for Women  \\
	Baramulla \\
	Jammu and Kashmir 193101, India}
\email{ahaijaz99@gmail.com}


\begin{abstract}
	Let $G_{n,m}=\overline{\C_n^{[m]}}$, where $\C_n^{[m]}$ denotes the
	closed $m$th power of the $n$-cycle. We study the graded Betti numbers
	and homological invariants of the edge ring of $G_{n,m}$ in the range
	$n\geq 3m+1$. These graphs form a natural family for the study of edge rings whose
	regularity can be compared explicitly with the induced matching number.
	In particular, for $n\geq4m+1$, the graph $G_{n,m}$ has induced matching
	number one, whereas its edge ring has regularity two.
	Our approach is based on a characterization of the homology of the
	induced subcomplexes of the independence complex
	$\Delta(G_{n,m})$. We introduce a family
	$\mathcal{S}_V(k,m)$ of vertex subsets characterized by their successive
	gaps around the cycle and show that, for $W\in\mathcal{S}_V(k,m)$, the
	induced subcomplex $\Delta[W]$ has the homotopy type of $\mathbb{S}^1$,
	whereas for $W\notin\mathcal{S}_V(k,m)$ all its positive-dimensional
	reduced homology groups vanish. Combining this characterization with
	Hochster's formula and an explicit enumeration of
	$\mathcal{S}_V(k,m)$, we obtain a closed formula for the graded Betti
	numbers in the second strand. We further determine the extremal Betti
	number, regularity, and projective dimension of the edge ring of $G_{n,m}$. Finally, we
	compute the $f$- and $h$-vectors of the independence complex and use
	the Hilbert series to determine the graded Betti numbers in the linear
	strand. The case $m=2$ recovers the corresponding results for complements of
	squares of cycles obtained in~\cite{RatherSquare}.
\end{abstract}

\maketitle

\section{Introduction}
\noindent
Let $R=\K[x_1,\ldots,x_n]$ be a polynomial ring over a field $\K$, and
let $G$ be a finite simple graph on the vertex set $\{1,\ldots,n\}$.
The edge ideal of $G$ is the square-free quadratic monomial ideal
\[
I(G)=\langle x_ix_j~|~\{i,j\}\in E_G\rangle\subseteq R.
\]
Throughout this paper, we refer to the quotient $R/I(G)$ as the edge
ring of $G$. The study of the minimal graded free resolution of
$R/I(G)$ provides a natural connection between the combinatorial
structure of $G$ and homological invariants such as its graded Betti
numbers, Castelnuovo--Mumford regularity and projective dimension.
Determining these invariants in terms of the combinatorial structure of
the underlying graph is a fundamental problem in the study of edge
ideals. Explicit formulas are known for several structured classes of
graphs; see, for example,
\cite{hatu,hawo,RatherCrown,RatherTriangular,vill,wood}.

An important problem in this direction is to understand the
Castelnuovo--Mumford regularity of edge rings. Fr\"oberg~\cite[Theorem~1]{frob}
characterized the graphs whose edge rings have regularity one: namely,
$\reg(G)=1$ if and only if $G$ is co-chordal. A complete combinatorial
characterization of graphs whose edge rings have regularity two is not
known, although important results are available for special classes;
see, for instance, Fern\'andez-Ramos and
Gimenez~\cite[Theorem~3.1]{fera}. On the other hand,
Katzman's description of the Betti numbers
$\beta_{i,2i}(G)$~\cite[Lemma~2.2]{katzman} implies that \( \reg(G)\geq\operatorname{im}(G), \)
where $\operatorname{im}(G)$ denotes the induced matching number of $G$.
It is therefore natural to study structured families of graphs for
which the relationship between regularity and induced matching number
can be determined explicitly, particularly when the inequality is
strict.

In this paper, we consider such a family arising from powers of cycles.
We work throughout the main part of the paper in the range
$n\geq3m+1$. Within this range, the induced matching number depends on
$n$: for $3m+1\leq n\leq4m$, the graph $G_{n,m}$ has induced matching
number two, whereas for $n\geq4m+1$ it has induced matching number one.
In contrast, we show that, we show that
\[
\reg(G_{n,m})=2
\]
throughout the entire range $n\geq3m+1$. Thus the subfamily
$n\geq4m+1$ provides a natural class of graphs for which the regularity
is strictly greater than the induced matching number.

The family $G_{n,m}$ is also particularly well suited to to an analysis
via  Hochster's
formula~\cite{hoch}, which expresses the graded Betti numbers of a
Stanley--Reisner ring in terms of the reduced homology of induced
subcomplexes. Indeed, an independent set of $G_{n,m}$ is precisely a clique
of $\C_n^{[m]}$, and therefore
\[
\Delta(G_{n,m})=\operatorname{Cl}(\C_n^{[m]}).
\]
Hence the computation of the graded Betti numbers of $R/I(G_{n,m})$
can be translated into the study of the reduced homology of induced
clique complexes of powers of cycles. Clique complexes of powers of cycles have been studied from a
topological point of view. In particular, Adamaszek~\cite{Adamaszek}
obtained homotopy-type results in terms of the order of the cycle and
its power. However, the
computation of graded Betti numbers requires more than the topology of
the full clique complex: Hochster's formula involves the reduced
homology of every induced subcomplex.

To obtain the required control over these induced subcomplexes, we
introduce a family $\mathcal{S}_V(k,m)$ of $k$-subsets of the vertex
set defined in terms of their successive gaps around the cycle.
We prove that
\[
W\in\mathcal{S}_V(k,m)
\Longrightarrow
\Delta[W]\simeq\mathbb S^1,
\]
whereas
\[
W\notin\mathcal{S}_V(k,m)
\Longrightarrow
\widetilde H_r(\Delta[W];\K)=0
~\text{for every }r>0.
\]

Thus the positive-dimensional homology of the induced subcomplexes is
determined entirely by a combinatorial condition on the successive
gaps of the vertices of $W$. Combining this characterization with
Hochster's formula and an explicit enumeration of
$\mathcal{S}_V(k,m)$, we obtain a closed formula for the graded Betti
numbers $\beta_{i,i+2}(G_{n,m})$. As consequences of this formula, we determine the extremal Betti
number, regularity and projective dimension of $G_{n,m}$.

The case $m=2$ specializes to complements of squares of cycles, studied
in our earlier work~\cite{RatherSquare}. The passage from $m=2$ to
arbitrary $m$ is not merely formal: for general $m$, the relevant clique
complexes are higher-dimensional, and the structure of their facets and
their intersections becomes considerably more involved. We describe
these facets in terms of successive gap sequences and use the Nerve
Lemma~\cite{Kozlov} to determine the homotopy type of the relevant
induced subcomplexes. In addition to the
second strand, we determine the $f$- and $h$-vectors of
$\Delta(G_{n,m})$ and use its Hilbert series to compute the graded
Betti numbers in the linear strand. Thus, in the range $n\geq3m+1$, all nontrivial graded Betti numbers of
$R/I(G_{n,m})$ are determined explicitly.

The paper is organized as follows. In Section~2, we recall the graph
theoretic, simplicial and homological notions used throughout the
paper. In Section~3, we first determine the induced matching number of
$G_{n,m}$ and then introduce the family $\mathcal{S}_V(k,m)$, determine
its cardinality and study the reduced homology of the corresponding
induced subcomplexes. We then compute the graded Betti
numbers in the second strand and determine the regularity and
projective dimension of $G_{n,m}$. In Section~4, we determine the $f$- and $h$-vectors of the independence
complex, use its Hilbert series to obtain the graded Betti numbers in
the linear strand, and conclude with a worked example illustrating the
formulas obtained in the paper.
\section{Preliminaries}
\noindent
Throughout this paper, $\K$ denotes a field and
$R=\K[x_1,\ldots,x_n]$ is the the standard graded polynomial ring.

Let $G=(V_G,E_G)$ be a finite simple graph. For $W\subseteq V_G$, the
induced subgraph of $G$ on $W$, denoted by $G[W]$, is the graph with vertex
set $W$ and edge set \[ E_{G[W]}=\{\{u,v\}\in E_G~|~u,v\in W\}. \]
The complement of $G$, denoted by $\overline G$, is the graph with vertex
set $V_G$ in which two distinct vertices are adjacent precisely when they
are not adjacent in $G$.

A path of length $\ell$ joining vertices $u$ and $v$ in $G$ is a
sequence \( u=u_0,u_1,\ldots,u_\ell=v \)
such that $\{u_{j-1},u_j\}\in E_G$ for $1\leq j\leq\ell$. If $G$ is
connected, the distance between $u$ and $v$, denoted by
$\mathrm d_G(u,v)$, is the minimum length of a path joining them. The
diameter of $G$ is
\[
\operatorname{diam}(G)
=
\max\{\mathrm d_G(u,v)\mid u,v\in V_G\}.
\] 

A subset $W\subseteq V_G$ is called a \emph{clique} of $G$ if every
two distinct vertices of $W$ are adjacent in $G$. On the other hand, a subset
$F\subseteq V_G$ is called \emph{independent} if no two distinct
vertices of $F$ are adjacent in $G$. 

Let $\C_n$ denote the cycle on the vertex set
$V=\{1,2,\ldots,n\}$, with the vertices arranged in cyclic order.
Whenever a vertex is indexed by an integer outside
$\{1,2,\ldots,n\}$, its index is reduced modulo $n$, with residue $0$
represented by $n$.
For example, $n+1$ represents the vertex $1$, $n+2$ represents the
vertex $2$, and $0$ represents the vertex $n$.

For vertices $a,b\in V$, the \emph{positively oriented arc} of $\C_n$
from $a$ to $b$ is the sequence of consecutive vertices encountered
when moving from $a$ to $b$ in the positive direction around the cycle.
Its length is the number of edges contained in this sequence. Thus, if $b\equiv a+r\pmod n$ with $0\leq r\leq n-1$, then the positively oriented arc from $a$ to $b$ has length $r$ and consists of the vertices $a,a+1,\ldots,a+r$,
where the indices are understood modulo $n$. More generally, by an \emph{arc of $\C_n$} we mean a connected subpath
of the cycle obtained by choosing one of the two directions between two
vertices. The length of an arc is the number of edges it contains.

\begin{definition}
	Let $G=(V_G,E_G)$ be a connected finite simple graph and let $m$ be a
	positive integer. The \emph{open $m$th power} of $G$, denoted by
	$G^{(m)}$, is the graph on $V_G$ in which two distinct vertices $u$
	and $v$ are adjacent if and only if
	\( \mathrm d_G(u,v)=m. \)
	The \emph{closed $m$th power} of $G$, denoted by $G^{[m]}$, is the
	graph on $V_G$ in which two distinct vertices $u$ and $v$ are adjacent
	if and only if
	\( \mathrm d_G(u,v)\leq m. \)
\end{definition}
For $i,j\in V(\C_n)$, the cyclic distance between $i$ and $j$ is defined by
\[
\mathrm d_{\C_n}(i,j)
=
\min\{|i-j|,\,n-|i-j|\}.
\]
Thus two distinct vertices $i$ and $j$ are adjacent in
$\C_n^{[m]}$ if and only if \(\mathrm d_{\C_n}(i,j)\leq m.\)
Throughout this paper, we study the complement of $\C_n^{[m]}$ and write \( G_{n,m}=\overline{\C_n^{[m]}}. \)
Hence \( V(G_{n,m})=\{1,2,\ldots,n\}, \)
and two distinct vertices $i$ and $j$ are adjacent in $G_{n,m}$ if
and only if \( \mathrm d_{\C_n}(i,j)>m. \)

When $m<n/2$, each vertex of $\C_n^{[m]}$ is adjacent to the $m$
vertices immediately preceding it and the $m$ vertices immediately
following it around the cycle. Hence, $G_{n,m}$ is
$(n-2m-1)$-regular and
\[
|E(G_{n,m})|
=
\frac{n(n-2m-1)}{2}.
\]
Moreover, cyclic rotations preserve adjacency, and hence $G_{n,m}$ is
vertex-transitive.

\begin{remark}
	Since \( 	\operatorname{diam}(\C_n)=\left\lfloor\frac{n}{2}\right\rfloor, \)
	we have $\C_n^{[m]}=K_n$ whenever
	$m\geq\lfloor n/2\rfloor$, where $K_n$ denotes the complete graph on
	$n$ vertices. Throughout our main results, we assume $n\geq3m+1$.
	This condition will be used in Section~3 to control the structure of
	the clique complexes arising from induced subgraphs of
	$\C_n^{[m]}$. The stronger condition $n\geq4m+1$ will arise only in
	connection with the induced matching number of $G_{n,m}$.
\end{remark}

%
\begin{definition}
	Let $V$ be a finite set. A \emph{simplicial complex} $\Delta$ on $V$
	is a collection of subsets of $V$ satisfying the following conditions:
	\begin{enumerate}
		\item $\{v\}\in\Delta$ for every $v\in V$;
		\item if $F\in\Delta$ and $G\subseteq F$, then $G\in\Delta$.
	\end{enumerate}
	The elements of $\Delta$ are called \emph{faces}. A face that is maximal
	with respect to inclusion is called a \emph{facet} of $\Delta$.
	The dimension of a face $F$ is defined by \( \dim F=|F|-1, \)
	and the dimension of $\Delta$ is
	\[
	\dim\Delta=\max\{\dim F~|~F\in\Delta\}.
	\]
\end{definition}
For $W\subseteq V$, the \emph{induced subcomplex} of $\Delta$ on $W$
is
\[
\Delta[W]=\{F\in\Delta\mid F\subseteq W\}.
\]

The \emph{independence complex} of a graph $G$, denoted by
$\Delta(G)$, is the simplicial complex whose faces are the independent
sets of $G$, whereas the \emph{clique complex}, denoted by
$\operatorname{Cl}(G)$, is the simplicial complex whose faces are the
cliques of $G$.

Since $G_{n,m}=\overline{\C_n^{[m]}}$, a subset of vertices is
independent in $G_{n,m}$ if and only if it is a clique in
$\C_n^{[m]}$. Consequently,
\[
\Delta(G_{n,m})
=
\operatorname{Cl}(\C_n^{[m]}).
\]

The graph $G_{12,3}=\overline{\C_{12}^{[3]}}$
and its independence complex are illustrated in Figure~\ref{fig:C12third}. Note that Its facets are
\[
F_i=\{i,i+1,i+2,i+3\},
\]
where \( 1\leq i\leq12 \) and the indices are taken modulo $12$.

\begin{figure}[htb]
	\centering
	\resizebox{\textwidth}{!}{%
		\begin{tikzpicture}[
			vertex/.style={
				circle,
				draw,
				fill=white,
				minimum size=4.5mm,
				inner sep=0pt,
				font=\scriptsize
			},
			edge/.style={line width=.45pt}
			]
			
			
			\begin{scope}[xshift=-4.2cm, scale=1.15]
				
				\def\R{2.15}
				
				\foreach \i in {1,...,12}
				{
					\pgfmathsetmacro{\ang}{90-(\i-1)*30}
					\node[vertex] (g\i) at (\ang:\R) {\i};
				}
				
				\foreach \i in {1,...,12}
				{
					\pgfmathtruncatemacro{\j}{mod(\i+3,12)+1}
					\draw[edge] (g\i)--(g\j);
				}
				
				\foreach \i in {1,...,12}
				{
					\pgfmathtruncatemacro{\j}{mod(\i+4,12)+1}
					\draw[edge] (g\i)--(g\j);
				}
				
				\foreach \i in {1,...,6}
				{
					\pgfmathtruncatemacro{\j}{\i+6}
					\draw[edge] (g\i)--(g\j);
				}
				
				\node at (0,-2.85)
				{(a) $G_{12,3}=\overline{\C_{12}^{[3]}}$};
				
			\end{scope}
			
			
			\begin{scope}[xshift=4cm]
				
				\tikzset{
					vtx/.style={
						circle,
						draw,
						fill=white,
						minimum size=4.5mm,
						inner sep=0pt,
						font=\scriptsize
					},
					bedge/.style={line width=.45pt},
					iedge/.style={black,line width=.45pt},
					dedge/.style={black,dashed,line width=.45pt}
				}
				
				%
				%
				
				\coordinate (v1)  at (-0.95, 2.45);
				\coordinate (v2)  at ( 0.00, 1.45);
				\coordinate (v3)  at ( 0.95, 2.45);
				
				\coordinate (v4)  at ( 2.45, 0.95);
				\coordinate (v5)  at ( 1.45, 0.00);
				\coordinate (v6)  at ( 2.45,-0.95);
				
				\coordinate (v7)  at ( 0.95,-2.45);
				\coordinate (v8)  at ( 0.00,-1.45);
				\coordinate (v9)  at (-0.95,-2.45);
				
				\coordinate (v10) at (-2.45,-0.95);
				\coordinate (v11) at (-1.45, 0.00);
				\coordinate (v12) at (-2.45, 0.95);

				
				\fill[gray!10]
				(v1)--(v3)--(v4)--(v6)--(v5)--(v2)--cycle;
				
				\fill[gray!14]
				(v4)--(v6)--(v7)--(v9)--(v8)--(v5)--cycle;
				
				\fill[gray!10]
				(v7)--(v9)--(v10)--(v12)--(v11)--(v8)--cycle;
				
				\fill[gray!14]
				(v10)--(v12)--(v1)--(v3)--(v2)--(v11)--cycle;

				
				\fill[white]
				(v2)--(v5)--(v8)--(v11)--cycle;
				
				\draw[bedge]
				(v2)--(v5)--(v8)--(v11)--cycle;

				
				\draw[iedge] (v1)--(v2)--(v3)--(v1);
				\draw[iedge] (v1)--(v4);
				\draw[iedge] (v2)--(v4);
				\draw[iedge] (v3)--(v4);
				
				\draw[iedge] (v2)--(v3);
				\draw[iedge] (v3)--(v4);
				\draw[iedge] (v4)--(v5);
				\draw[iedge] (v2)--(v5);
				\draw[dedge] (v3)--(v5);
				\draw[dedge] (v2)--(v4);
				
				\draw[iedge] (v3)--(v4);
				\draw[iedge] (v4)--(v5)--(v6)--(v4);
				\draw[iedge] (v3)--(v6);
				\draw[dedge] (v3)--(v5);

				
				\draw[iedge] (v4)--(v5)--(v6)--(v4);
				\draw[iedge] (v4)--(v7);
				\draw[iedge] (v5)--(v7);
				\draw[iedge] (v6)--(v7);
				
				\draw[iedge] (v6)--(v7);
				\draw[iedge] (v7)--(v8);
				\draw[iedge] (v5)--(v8);
				\draw[dedge] (v6)--(v8);
				\draw[dedge] (v5)--(v7);
				
				\draw[iedge] (v7)--(v8)--(v9)--(v7);
				\draw[iedge] (v6)--(v9);
				\draw[dedge] (v6)--(v8);

				
				\draw[iedge] (v7)--(v8)--(v9)--(v7);
				\draw[iedge] (v7)--(v10);
				\draw[iedge] (v8)--(v10);
				\draw[iedge] (v9)--(v10);
				
				\draw[iedge] (v9)--(v10);
				\draw[iedge] (v10)--(v11);
				\draw[iedge] (v8)--(v11);
				\draw[dedge] (v9)--(v11);
				\draw[dedge] (v8)--(v10);
				
				\draw[iedge] (v10)--(v11)--(v12)--(v10);
				\draw[iedge] (v9)--(v12);
				\draw[dedge] (v9)--(v11);

				
				\draw[iedge] (v10)--(v11)--(v12)--(v10);
				\draw[iedge] (v10)--(v1);
				\draw[iedge] (v11)--(v1);
				\draw[iedge] (v12)--(v1);
				
				\draw[iedge] (v12)--(v1);
				\draw[iedge] (v1)--(v2);
				\draw[iedge] (v11)--(v2);
				\draw[dedge] (v12)--(v2);
				\draw[dedge] (v11)--(v1);
				
				\draw[iedge] (v1)--(v2)--(v3)--(v1);
				\draw[iedge] (v12)--(v3);
				\draw[dedge] (v12)--(v2);

				
				\draw[bedge]
				(v1)--(v3)--(v4)--(v6)--(v7)--(v9)--
				(v10)--(v12)--cycle;

				
				\foreach \i in {1,...,12}
				{
					\node[vtx] at (v\i) {\i};
				}

				
				\node[font=\scriptsize, rotate=-45] at (2,2)
				{$F_1,F_2,F_3$};
				
				\node[font=\scriptsize,rotate=-90] at (2.90,-0.05)
				{$F_4,F_5,F_6$};
				
				\node[font=\scriptsize, rotate=-45] at (-2,-2)
				{$F_7,F_8,F_9$};
				
				\node[font=\scriptsize,rotate=90] at (-2.90,0.05)
				{$F_{10},F_{11},F_{12}$};

				\node at (0,-3.4)
				{(b) $\Delta(G_{12,3})$};
				
			\end{scope}
	\end{tikzpicture}}
	\caption{The graph $G_{12,3}=\overline{\C_{12}^{[3]}}$
		and its independence complex.}
	
	\label{fig:C12third}
	
\end{figure}
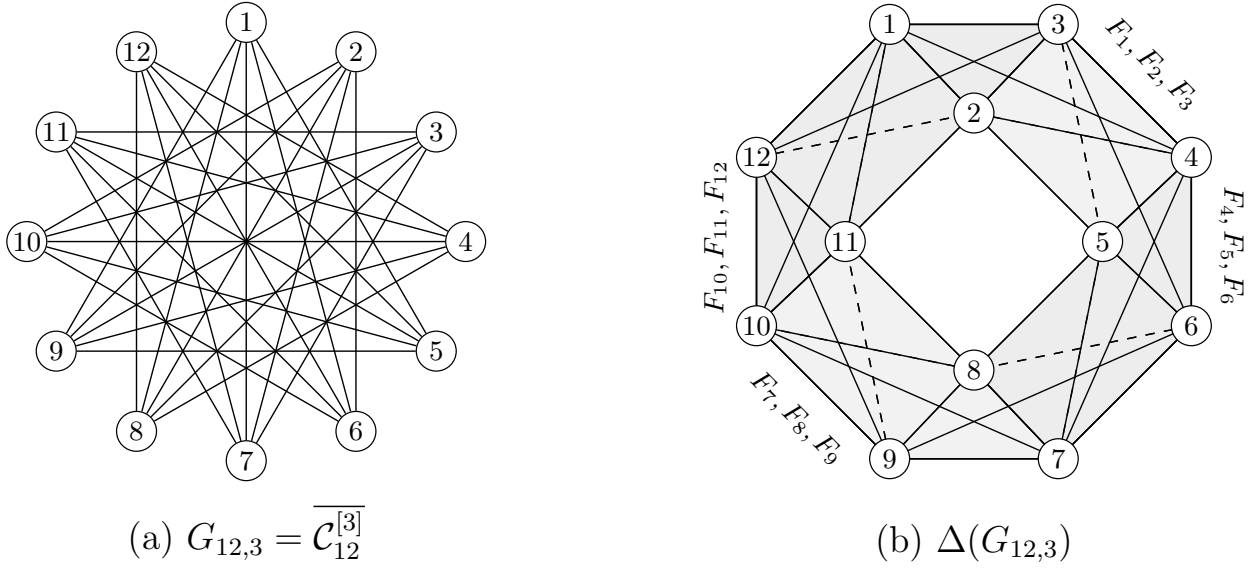

Let $\Gamma$ be a simplicial complex on the vertex set
$\{1,2,\ldots,n\}$. The Stanley--Reisner ideal of $\Gamma$ is the
square-free monomial ideal
\[
I_{\Gamma}
=
\left\langle
x_{i_1}x_{i_2}\cdots x_{i_s}
\;\middle|\;
\{i_1,i_2,\ldots,i_s\}\notin\Gamma
\right\rangle
\subseteq R.
\]
The quotient ring \( \K[\Gamma]=R/I_{\Gamma} \)
is called the Stanley--Reisner ring of $\Gamma$.


Since the minimal non-faces of the independence complex $\Delta(G)$ are
precisely the edges of $G$, we have \( I_{\Delta(G)}=I(G). \) 
Hence, \( \K[\Delta(G)]\cong R/I(G). \)
As in the Introduction, we refer to $R/I(G)$ as the edge ring of $G$.
 Let $M$ be a finitely generated graded $R$-module. A minimal graded
free resolution of $M$ has the form
\[
0\longrightarrow
\bigoplus_j R(-j)^{\beta_{p,j}(M)}
\longrightarrow\cdots\longrightarrow
\bigoplus_j R(-j)^{\beta_{1,j}(M)}
\longrightarrow
\bigoplus_j R(-j)^{\beta_{0,j}(M)}
\longrightarrow M\longrightarrow0,
\]
where $R(-j)$ denotes the graded free module obtained from $R$ by a
degree shift of $j$, so that $R(-j)_d=R_{d-j}$. The integers
$\beta_{i,j}(M)$ are called the \emph{graded Betti numbers} of $M$;
equivalently, $\beta_{i,j}(M)$ is the number of minimal generators of
degree $j$ in the $i$th free module of the resolution.

The \emph{Castelnuovo--Mumford regularity} and the
\emph{projective dimension} of $M$ are defined, respectively, by
\[
\reg(M)=\max\{j-i:\beta_{i,j}(M)\neq0\}
\]
and
\[
\pd(M)=\max\{i:\beta_{i,j}(M)\neq0
\text{ for some }j\}.
\]
For a graph $G$, we use the shorthand notation \( \beta_{i,j}(G)=\beta_{i,j}(R/I(G)),
~\reg(G)=\reg(R/I(G))\), \(
~\pd(G)=\pd(R/I(G)). \)
For the remainder of the paper, we write
\[
\Delta=\Delta(G_{n,m})
=\operatorname{Cl}(\C_n^{[m]}).
\]

The following standard form of the Nerve Lemma will be used; see,
for example,~\cite{Kozlov}.
\begin{theorem}[Nerve Lemma]\label{thm:nerve}
	Let $\Gamma$ be a simplicial complex and let
	\( \Gamma=\Gamma_1\cup\Gamma_2\cup\cdots\cup\Gamma_s \)
	be a cover by subcomplexes. Suppose that every nonempty finite
	intersection \( \Gamma_{j_1}\cap\Gamma_{j_2}\cap\cdots\cap\Gamma_{j_r} \)
	is contractible. Then $\Gamma$ is homotopy equivalent to the nerve of
	the cover.
\end{theorem}

For a simplicial complex $\Gamma$,
$\widetilde H_r(\Gamma;\K)$ denotes its $r$th reduced simplicial
homology group with coefficients in $\K$. We recall Hochster's formula~\cite{hoch},
which will be used repeatedly throughout the paper.

\begin{theorem}[Hochster's Formula]\label{hochster}
	Let $\Gamma$ be a simplicial complex on the vertex set
	$V=\{1,2,\ldots,n\}$. Then
	\[
	\beta_{i,j}(\K[\Gamma])
	=
	\sum_{\substack{W\subseteq V\\ |W|=j}}
	\dim_{\K}\widetilde H_{j-i-1}(\Gamma[W];\K).
	\]
\end{theorem}
Applying Theorem~\ref{hochster} to
$\Delta=\Delta(G_{n,m})$, we obtain
\[
\beta_{i,j}(G_{n,m})
=
\sum_{\substack{W\subseteq V\\ |W|=j}}
\dim_{\K}
\widetilde H_{j-i-1}(\Delta[W];\K).
\]
Thus the computation of the graded Betti numbers of $G_{n,m}$ reduces to
the computation of the reduced homology of the induced subcomplexes of
$\Delta$.

\section{Complements of closed powers of cycles}
\noindent
We begin by determining the induced matching number of $G_{n,m}$,
thereby clarifying the motivation described in the Introduction. We
then introduce the analogue of the family $\mathcal{S}_V(k,2)$ used
for complements of squares of cycles.
\begin{proposition}\label{prop:im}
	Let $n\geq4m+1$. Then \( \operatorname{im}(G_{n,m})=1. \)
\end{proposition}

\begin{proof}
	Since $G_{n,m}$ has an edge, we have
	$\operatorname{im}(G_{n,m})\geq1$. Suppose, contrary to the assertion,
	that
	\[
	\{a,b\},\{c,d\}\in E(G_{n,m})
	\]
	form an induced matching. Since
	$G_{n,m}=\overline{\C_n^{[m]}}$, we have \( \mathrm d_{\C_n}(a,b)>m \) and  \( \mathrm d_{\C_n}(c,d)>m. \)On the other hand, there are no edges of $G_{n,m}$ joining an endpoint
	of $\{a,b\}$ to an endpoint of $\{c,d\}$. Hence
	\[
	\mathrm d_{\C_n}(a,c),
	\mathrm d_{\C_n}(a,d),
	\mathrm d_{\C_n}(b,c),
	\mathrm d_{\C_n}(b,d)
	\leq m.
	\]
	
	Let \( s=\mathrm d_{\C_n}(a,b). \)
	Since $s>m$, we have $s\geq m+1$. Moreover,
	$s\leq\lfloor n/2\rfloor$. Because $n\geq4m+1$, the longer arc between
	$a$ and $b$ has length
	\[
	n-s\geq\left\lceil\frac n2\right\rceil>2m.
	\]
Since the longer arc joining $a$ and $b$ has length greater than
$2m$, no vertex in its interior can be within cyclic distance at most
$m$ from both $a$ and $b$. Hence, every vertex lying within
distance at most $m$ from both $a$ and $b$ must lie on the shorter arc
joining $a$ and $b$.
	
	Along this shorter arc, the vertices whose distances from both $a$ and
	$b$ are at most $m$ form an arc of length \( 2m-s. \) Since $s\geq m+1$, we have \( 2m-s\leq m-1. \)
	Both $c$ and $d$ belong to this arc. Therefore
	\[
	\mathrm d_{\C_n}(c,d)\leq m-1<m,
	\]
	contradicting the fact that $\{c,d\}\in E(G_{n,m})$. Hence $G_{n,m}$ contains no induced matching of size two, and therefore
	\( \operatorname{im}(G_{n,m})=1. \)
\end{proof}
\begin{remark}\label{rem:im}
	The bound $n\geq4m+1$ in Proposition~\ref{prop:im} is sharp. Indeed, suppose \( 3m+1\leq n\leq4m \)
	and write \( n=3m+r, \) where \( 1\leq r\leq m. \)

	Consider the four vertices
	\(
	1,~ r+1,~ m+r+1,\) and \(2m+r+1.
	\)
	The pairs
	\(
	\{1,m+r+1\}
	~\text{and}~
	\{r+1,2m+r+1\}
	\)
	are edges of $G_{n,m}$, since the cyclic distance between the vertices
	of each pair is greater than $m$. On the other hand, each of the four
	cross-pairs has cyclic distance at most $m$. Hence these two edges form
	an induced matching in $G_{n,m}$.
	Thus, \( \operatorname{im}(G_{n,m})\geq2 \)
	whenever \( 3m+1\leq n\leq4m \).

\end{remark}

\begin{definition}\label{def:SV}
	Let $1\leq k\leq n$ and let \( V=\{1,2,\ldots,n\}. \)
	We define $\mathcal{S}_V(k,m)$ to be the collection of all $k$-element
	subsets
	\[
	W=\{i_1,i_2,\ldots,i_k\}\subseteq V,
	\]
	where \( 1\leq i_1<i_2<\cdots<i_k\leq n, \) such that \( 1\leq i_{t+1}-i_t\leq m,~ 1\leq t\leq k-1 \)
	and \( 1\leq n+i_1-i_k\leq m. \)
	Equivalently, if $i_{k+1}=i_1$, then for each
	$1\leq t\leq k$ there exists a unique integer
	$r_t\in\{1,2,\ldots,m\}$ such that
	\[
	i_{t+1}\equiv i_t+r_t\pmod n.
	\]
\end{definition}

\begin{remark}\label{rem:gaps}
Let \( W=\{i_1,i_2,\ldots,i_k\}\in\mathcal{S}_V(k,m), \) where \( 1\leq i_1<i_2<\cdots<i_k\leq n. \) By Definition~\ref{def:SV}, the positive gap from each $i_t$ to the
successive vertex $i_{t+1}$ of $W$, measured in the positive direction
around $\C_n$, is at most $m$. For $1\leq t\leq k-1$, this gap is \( r_t=i_{t+1}-i_t, \)
whereas the gap from $i_k$ back to $i_1$ is \( r_k=n+i_1-i_k. \)
Thus \( 1\leq r_t\leq m, \)
for all \( 1\leq t\leq k. \)	

	The integers $r_1,r_2,\ldots,r_k$ measure the successive gaps between the
	vertices of $W$. Starting from $i_1$ and successively moving through
	$i_2,i_3,\ldots,i_k$ and then back to $i_1$ amounts to traversing the
	whole cycle exactly once. Hence the sum of these gaps is $n$. Indeed,
	\[
	\begin{aligned}
		r_1+r_2+\cdots+r_k
		=(i_2-i_1)+(i_3-i_2)+\cdots+(i_k-i_{k-1})
		+(n+i_1-i_k)
		=n.
	\end{aligned}
	\]
	Since each $r_t\leq m$, we also have
	\[
	n=r_1+r_2+\cdots+r_k\leq
	\underbrace{m+m+\cdots+m}_{k\text{ times}}=mk.
	\]
It follows that, \( \mathcal{S}_V(k,m)=\varnothing \) whenever \( k<\left\lceil\frac{n}{m}\right\rceil. \)
Thus, although Definition~\ref{def:SV} is meaningful for every
$1\leq k\leq n$, the family can be nonempty only if
\( k\geq\left\lceil\frac{n}{m}\right\rceil. \)
\end{remark}

\begin{remark}\label{rem:binary}
For $W\subseteq V$, define its characteristic sequence
$s_W=(s_1,\ldots,s_n)\in\{0,1\}^n$ by
\[
s_j=
\begin{cases}
1,&j\in W,\\
0,&j\notin W.
\end{cases}
\]
Then $W\in\mathcal{S}_V(k,m)$ if and only if $|\operatorname{supp}(s_W)|=k$
and there is no $j\in\{1,\ldots,n\}$ for which
\[
s_j=s_{j+1}=\cdots=s_{j+m-1}=0,
\]
where the subscripts are taken modulo $n$. Indeed, if $r_t$ is the gap between two successive vertices of $W$, then
there are exactly $r_t-1$ consecutive zero entries between the
corresponding two entries equal to $1$. Thus the condition $r_t\leq m$
is equivalent to saying that no block of $m$ consecutive entries of
$s_W$ consists entirely of zeros.
\end{remark}

\begin{lemma}\label{lem:rotation}
	Let $W\subseteq V=\{1,2,\ldots,n\}$, and let \( s_W=(s_1,s_2,\ldots,s_n) \) be its characteristic sequence. Suppose that \( 	s_{W'}
	=
	(s_{\ell},s_{\ell+1},\ldots,s_n,s_1,\ldots,s_{\ell-1}) \)
	is a rotation of $s_W$ for some $1\leq \ell\leq n$. Then
	\[
	\Delta[W']\cong\Delta[W].
	\]
	In particular, \( 	\widetilde H_r(\Delta[W'];\K)
	\cong
	\widetilde H_r(\Delta[W];\K) \)
	for every $r$.
\end{lemma}

\begin{proof}
	Consider the map
	\[
	\rho_{\ell}:V\longrightarrow V,~
	\rho_{\ell}(j)\equiv j-\ell+1\pmod n.
	\]
	Since cyclic distance is invariant under rotation, we have
	\[
	\mathrm{d}_{\C_n}(u,v)
	=
	\mathrm{d}_{\C_n}\bigl(\rho_{\ell}(u),\rho_{\ell}(v)\bigr)
	\]
	for all $u,v\in V$. Consequently, \( \{u,v\}\in E(\C_n^{[m]}) \) if and only if \( \{\rho_{\ell}(u),\rho_{\ell}(v)\}\in E(\C_n^{[m]}).\) Thus $\rho_{\ell}$ is an automorphism of $\C_n^{[m]}$, and hence also
	of
$G_{n,m}=\overline{\C_n^{[m]}}$.	By the definition of $W'$, the map $\rho_{\ell}$ sends $W$ onto $W'$.
	Therefore its restriction to $W$ induces a simplicial isomorphism
	\[
	\Delta[W]\cong\Delta[W'].
	\]
	It follows immediately that \( \widetilde H_r(\Delta[W];\K)
	\cong
	\widetilde H_r(\Delta[W'];\K) \)
	for every $r$.
\end{proof}

In the following lemma, we shall determine the cardinality of $\mathcal{S}_V(k,m)$.

\begin{lemma}\label{lem:count}
	For \( n\geq k\geq\left\lceil\frac{n}{m}\right\rceil, \)
	we have
	\[
	\left|\mathcal{S}_V(k,m)\right|
	=
	\frac{n}{k}
	[z^{\,n-k}]
	(1+z+\cdots+z^{m-1})^k.
	\]
	Equivalently,
	\[
	\left|\mathcal{S}_V(k,m)\right|
	=
	\frac{n}{k}
	\sum_{q=0}^{\left\lfloor\frac{n-k}{m}\right\rfloor}
	(-1)^q
	\binom{k}{q}
	\binom{n-mq-1}{k-1}.
	\]
\end{lemma}

\begin{proof}
	Let \(W=\{i_1,i_2,\ldots,i_k\}\in\mathcal{S}_V(k,m),\) where \(1\leq i_1<i_2<\cdots<i_k\leq n.\)
	As observed in Remark~\ref{rem:gaps}, associated to $W$ are the integers \( r_t=i_{t+1}-i_t, 1\leq t\leq k-1, \) and \( r_k=n+i_1-i_k. \)
	They satisfy \(1\leq r_t\leq m\), \( 1\leq t\leq k \) and \( r_1+r_2+\cdots+r_k=n. \)
	
	Put \( a_t=r_t-1 \) for \( 1\leq t\leq k \). Then
	\( 0\leq a_t\leq m-1 \)
	and \( a_1+a_2+\cdots+a_k=n-k. \)
	Thus, we first need to count the ordered $k$-tuples
	$(a_1,a_2,\ldots,a_k)$ satisfying these conditions. For this purpose, introduce an auxiliary variable $z$. Since each $a_t$
	can take any of the values $0,1,\ldots,m-1$, these possibilities are
	represented by the polynomial
	\[
	1+z+z^2+\cdots+z^{m-1}.
	\]
	Hence, the possible ordered $k$-tuples
	$(a_1,\ldots,a_k)$ are represented by
	\( (1+z+z^2+\cdots+z^{m-1})^k. \)
	Indeed, on choosing the term $z^{a_t}$ from the $t$th factor, the
	corresponding term in the product is
	\[
	z^{a_1}z^{a_2}\cdots z^{a_k}
	=
	z^{a_1+a_2+\cdots+a_k}.
	\]
	It follows that the number of ordered $k$-tuples satisfying \( a_1+\cdots+a_k=n-k \)
	is precisely the coefficient of $z^{n-k}$ in \( (1+z+z^2+\cdots+z^{m-1})^k. \)
	If $[z^r]f(z)$ denotes the coefficient of $z^r$ in a polynomial $f(z)$, this number is therefore \( [z^{\,n-k}](1+z+z^2+\cdots+z^{m-1})^k. \)
	
	We now relate these ordered $k$-tuples to the members of
	$\mathcal{S}_V(k,m)$. Fix an ordered $k$-tuple
	\( (a_1,a_2,\ldots,a_k) \) satisfying \( 0\leq a_t\leq m-1\) and \(a_1+\cdots+a_k=n-k, \)
	and choose a vertex $v\in V$. Put $i_1=v$ and define successively
	\[
	i_{t+1}\equiv i_t+a_t+1\pmod n,
	\]
	for all \( 1\leq t\leq k. \) Since each $a_t+1$ is positive and \( \sum_{t=1}^{k}(a_t+1)=n, \)
	we have, for every $1\leq s<k$,
\( 	0<
\sum_{t=1}^{s}(a_t+1)
<n\)

	Therefore, no proper partial sum is congruent to $0$ modulo $n$, and
	hence the vertices \( i_1,i_2,\ldots,i_k \)
	are pairwise distinct. After the $k$th step we return to the initial
	vertex, since
	\[
	i_{k+1}
	\equiv
	i_1+\sum_{t=1}^{k}(a_t+1)
	\equiv i_1\pmod n.
	\]
	Moreover, \( 1\leq a_t+1\leq m \) for every $t$, and therefore the resulting $k$-element set belongs to
	$\mathcal{S}_V(k,m)$. Thus an ordered $k$-tuple $(a_1,\ldots,a_k)$ together with a distinguished
	starting vertex determines a unique pair $(W,v)$, where \( W\in\mathcal{S}_V(k,m) \) and \( v\in W \).
	
	Conversely, every such pair determines its ordered sequence of gaps
	uniquely. Thus this correspondence is a bijection. Therefore, the total number
	of pairs $(W,v)$, where \( W\in\mathcal{S}_V(k,m) \) and \( v\in W \) is \( n[z^{\,n-k}]
	(1+z+z^2+\cdots+z^{m-1})^k. \)
	On the other hand, every set $W\in\mathcal{S}_V(k,m)$ has exactly $k$
	vertices, and any one of these vertices can be chosen as the initial
	vertex. Therefore every $W$ is counted exactly $k$ times. It follows that
	\[
	k\left|\mathcal{S}_V(k,m)\right|
	=
	n[z^{\,n-k}]
	(1+z+z^2+\cdots+z^{m-1})^k,
	\]
	and hence
	\begin{equation}\label{coeff}
		\left|\mathcal{S}_V(k,m)\right|
		=
		\frac{n}{k}
		[z^{\,n-k}]
		(1+z+z^2+\cdots+z^{m-1})^k.
	\end{equation}
	
	It remains to evaluate this coefficient. Note that
	\( 1+z+z^2+\cdots+z^{m-1}
	=
	\frac{1-z^m}{1-z}, \) and hence
	\( 	(1+z+z^2+\cdots+z^{m-1})^k
	=
	(1-z^m)^k(1-z)^{-k}. \) Using the binomial theorem,\( (1-z^m)^k =
	\sum_{q=0}^{k}
	(-1)^q\binom{k}{q}z^{mq}, \) and \( (1-z)^{-k}
	=
	\sum_{\ell=0}^{\infty}
	\binom{k+\ell-1}{k-1}z^\ell. \) Therefore,
	\[
	(1+z+\cdots+z^{m-1})^k
	=
	\sum_{q=0}^{k}\sum_{\ell=0}^{\infty}
	(-1)^q
	\binom{k}{q}
	\binom{k+\ell-1}{k-1}
	z^{mq+\ell}.
	\]
	
	Finally, to obtain the coefficient of $z^{n-k}$, we must have \( mq+\ell=n-k, \) or equivalently,
	\( \ell=n-k-mq. \)
	Since $\ell\geq0$, we necessarily have \( 0\leq q\leq
	\left\lfloor\frac{n-k}{m}\right\rfloor. \)
	Thus
	\[
	[z^{\,n-k}]
	(1+z+\cdots+z^{m-1})^k
	=
	\sum_{q=0}^{\left\lfloor\frac{n-k}{m}\right\rfloor}
	(-1)^q
	\binom{k}{q}
	\binom{n-mq-1}{k-1}.
	\]
	Substituting this into  (\ref{coeff}), we have
	\[
	\left|\mathcal{S}_V(k,m)\right|
	=
	\frac{n}{k}
	\sum_{q=0}^{\left\lfloor\frac{n-k}{m}\right\rfloor}
	(-1)^q
	\binom{k}{q}
	\binom{n-mq-1}{k-1}.
	\]
\end{proof}
\begin{remark}
	For $m=2$, Lemma~\ref{lem:count} gives
	\[
	\left|\mathcal{S}_V(k,2)\right|
	=
	\frac{n}{k}[z^{\,n-k}](1+z)^k
	=
	\frac{n}{k}\binom{k}{n-k}.
	\]
	Using Pascal's identity,
	\[
	\frac{n}{k}\binom{k}{n-k}
	=
	\binom{k}{n-k}
	+
	\binom{k-1}{n-k-1}.
	\]
Hence, for $m=2$, the formula reduces to the corresponding formula
obtained for complements of squares of cycles in~\cite{RatherSquare}.
\end{remark}

\subsection{Homology of the induced subcomplexes}

We now study the reduced homology of the induced subcomplexes
$\Delta[W]$. The argument naturally separates into two cases according
as $W$ belongs to $\mathcal{S}_V(k,m)$ or not. We begin with a basic
observation concerning closed powers of paths, which will be used in
the latter case.

\begin{lemma}\label{lem:pathpower}
	Let $P_s$ be a path on $s$ vertices. Then
	\( \operatorname{Cl}(P_s^{[m]}) \)
	is contractible.
\end{lemma}

\begin{proof}
	Label the vertices of $P_s$ by \( 1,2,\ldots,s \)
	in their natural order along the path. To each vertex $i$, associate the closed interval \( I_i=[i,i+m]. \)	For distinct vertices $i$ and $j$, we have
	\[
	I_i\cap I_j\neq\varnothing
	\Longleftrightarrow
	|i-j|\leq m.
	\]
	On the other hand, \( \mathrm{d}_{P_s}(i,j)=|i-j|. \)
	Therefore
	\[
	I_i\cap I_j\neq\varnothing
	\Longleftrightarrow
	\mathrm{d}_{P_s}(i,j)\leq m
	\Longleftrightarrow
	\{i,j\}\in E(P_s^{[m]}).
	\]
	Thus $P_s^{[m]}$ is an interval graph.
	
	Since every interval graph is chordal, $P_s^{[m]}$ is chordal. Moreover,
	$P_s^{[m]}$ is connected. The clique complex of a connected chordal
	graph is contractible. Hence \( \operatorname{Cl}(P_s^{[m]}) \)
	is contractible.
\end{proof}

\begin{lemma}\label{lem:contractible}
	Let \( W=\{i_1,i_2,\ldots,i_k\}\subseteq V. \)
	If \( W\notin\mathcal{S}_V(k,m), \)
	then
	\[
	\widetilde H_r(\Delta[W];\K)=0
	\]
	for every  \( r>0 \)
\end{lemma}

\begin{proof}
	Since $W\notin\mathcal{S}_V(k,m)$, at least one successive gap between
	vertices of $W$ is greater than $m$. Hence there is a block of at least
	$m$ consecutive vertices of $\C_n$ containing no vertex of $W$.
	By Lemma~\ref{lem:rotation}, after applying a cyclic rotation if
	necessary, we may assume that
	\[
	W\cap\{1,2,\ldots,m\}=\varnothing.
	\]
	
	Delete the vertices $1,2\ldots,m$ from $\C_n$. The remaining vertices \( m+1,m+2,\ldots,n \)
	occur consecutively on a path, which we denote by $P$. We claim that
	for vertices $u,v\in W$, \( \{u,v\}\in E(\C_n^{[m]}) \)
	if and only if $u$ and $v$ are at distance at most $m$ in $P$.
	
	Indeed, any path in $\C_n$ joining two vertices of $W$ and passing
	through the deleted block must enter the block at one end and leave it
	at the other. Since the block contains $m$ consecutive deleted
	vertices, such a path has length at least $m+1$. Therefore no path of
	length at most $m$ joining vertices of $W$ can pass through the deleted
	block. Therefore, whenever \( \mathrm{d}_{\C_n}(u,v)\leq m, \)
	the shortest path from $u$ to $v$ is entirely contained in $P$, and
	hence \( 	\mathrm{d}_{\C_n}(u,v)=\mathrm{d}_P(u,v). \)
	Thus \( \C_n^{[m]}[W]=P^{[m]}[W]. \)
	
	As observed in the proof of Lemma~\ref{lem:pathpower},
	$P^{[m]}$ is an interval graph.
	Every induced subgraph of an interval graph is again an interval graph;
	hence every connected component of $\C_n^{[m]}[W]$ is a connected
	chordal graph. The clique complex of a connected chordal graph is
	contractible. Therefore every connected component of \( \Delta[W]
	=
	\operatorname{Cl}\bigl(\C_n^{[m]}[W]\bigr) \)
	is contractible. It follows that $\Delta[W]$ has no reduced homology in positive
	dimensions. Hence
	\[
	\widetilde H_r(\Delta[W];\K)=0
	\]
	for every \(r>0\).
\end{proof}

\begin{lemma}\label{lem:arc}
	Let $n\geq3m+1$. Then every clique of $\C_n^{[m]}$ is contained in an
	arc of $\C_n$ of length at most $m$.
\end{lemma}

\begin{proof}
	Fix a vertex $u\in F$. Since $F$ is a clique of $\C_n^{[m]}$, every
	vertex of $F$ is at distance at most $m$ from $u$ in $\C_n$. Therefore,
	each vertex of $F$ lies either within $m$ steps from $u$ in the positive
	direction around the cycle or within $m$ steps from $u$ in the negative
	direction.
	
	Let $p$ be the largest number of steps from $u$ in the positive direction
	to a vertex of $F$, and let $q$ be the largest number of steps from $u$
	in the negative direction to a vertex of $F$. 
	
	Thus
	\( 0\leq p\leq m \) and \( 0\leq q\leq m. \)
	If $p=0$ or $q=0$, then all the vertices of $F$ lie on one side of $u$
	and are therefore contained in an arc of length at most $m$. Now suppose that $p>0$ and $q>0$. Let $v,w\in F$ be vertices such that
	$v$ lies $p$ steps from $u$ in the positive direction and $w$ lies $q$
	steps from $u$ in the negative direction. The two arcs of $\C_n$ joining
	$v$ and $w$ have lengths \( p+q \) and \( n-(p+q). \)
	Since $p,q\leq m$, we have \( p+q\leq 2m. \)
	Using $n\geq3m+1$, it follows that \( 	n-(p+q)
	\geq
	3m+1-2m
	=
	m+1. \) Hence the second arc joining $v$ and $w$ has length strictly greater
	than $m$.
	
	On the other hand, $v$ and $w$ belong to the clique $F$, so they are
	adjacent in $\C_n^{[m]}$. It follows that,
	\( \mathrm{d}_{\C_n}(v,w)\leq m. \)
	Since the arc of length $n-(p+q)$ has length greater than $m$, the
	distance between $v$ and $w$ must be realized by the other arc. Therefore
	\( p+q\leq m. \)
	
	The arc beginning at $w$, passing through $u$, and ending at $v$ has
	length $p+q\leq m$, and by the definitions of $p$ and $q$, it contains
	every vertex of $F$. Hence all the vertices of $F$ lie in an arc of
	$\C_n$ of length at most $m$.
\end{proof}
\begin{lemma}\label{lem:facetsW}
	Let $n\geq3m+1$ and let \( 	W=\{i_1,\ldots,i_k\}\in\mathcal S_V(k,m), \)
	with successive gaps $r_1,\ldots,r_k$. For each $t$, let
	\[
	s_t=\max\left\{1\leq s\leq k-1:
	r_t+r_{t+1}+\cdots+r_{t+s-1}\leq m\right\},
	\]
	where the subscripts are taken modulo $k$, and set
	\( F_t=\{i_t,i_{t+1},\ldots,i_{t+s_t}\}. \)
	Then \( \Delta[W]=\bigcup_{t=1}^k\langle F_t\rangle. \)
	Thus, the facets of $\Delta[W]$ are precisely the
	inclusion-maximal members among $F_1,\ldots,F_k$.
\end{lemma}

\begin{proof}
Since $W\in\mathcal S_V(k,m)$, Remark~\ref{rem:gaps} gives \( k\geq\left\lceil\frac{n}{m}\right\rceil\geq4. \)
Hence $k-1\geq1$. Moreover, every successive gap satisfies
$1\leq r_t\leq m$, so $s=1$ is admissible in the definition of $s_t$.
Thus $s_t$ is well defined for every $t$.
	We first show that each $F_t$ is a face of $\Delta[W]$. By the
	definition of $s_t$, \( r_t+r_{t+1}+\cdots+r_{t+s_t-1}\leq m. \)
	The left-hand side is the length of the positively oriented arc of
	$\C_n$ beginning at $i_t$ and ending at $i_{t+s_t}$. Thus all the
	vertices of \( F_t=\{i_t,i_{t+1},\ldots,i_{t+s_t}\} \)
	lie on an arc of length at most $m$.
	
	Let $i_{t+p}$ and $i_{t+q}$ be two vertices of $F_t$, where
	$0\leq p<q\leq s_t$. The positively oriented arc from $i_{t+p}$ to
	$i_{t+q}$ has length \( r_{t+p}+r_{t+p+1}+\cdots+r_{t+q-1}, \)
	which is at most \( r_t+r_{t+1}+\cdots+r_{t+s_t-1}\leq m. \)
	Hence \( \mathrm d_{\C_n}(i_{t+p},i_{t+q})\leq m. \)
	Therefore every two vertices of $F_t$ are adjacent in
	$\C_n^{[m]}[W]$, so $F_t$ is a clique of $\C_n^{[m]}[W]$. Since
	\[
	\Delta[W]=\operatorname{Cl}\bigl(\C_n^{[m]}[W]\bigr),
	\]
	it follows that $F_t\in\Delta[W]$.
	
	We now prove the converse inclusion. Let $F\in\Delta[W]$. then $F$ is
	a clique of $\C_n^{[m]}[W]$. By Lemma~\ref{lem:arc}, all the vertices
	of $F$ are contained in an arc of $\C_n$ of length at most $m$.
	
	Let $i_t$ be the first vertex of $F$ encountered on such an arc, and
	let $i_{t+s}$ be the last one, where the indices are taken modulo $k$.
	The set $F$ need not contain every vertex
	$i_t,i_{t+1},\ldots,i_{t+s}$, but every vertex of $F$ lies among these
	vertices of $W$. Since the arc from $i_t$ to $i_{t+s}$ has length at
	most $m$, we have \( r_t+r_{t+1}+\cdots+r_{t+s-1}\leq m. \)
	By the maximality in the definition of $s_t$, this implies
	$s\leq s_t$. Consequently, \( 	F\subseteq
	\{i_t,i_{t+1},\ldots,i_{t+s}\}
	\subseteq
	\{i_t,i_{t+1},\ldots,i_{t+s_t}\}
	=
	F_t. \)
	
	Thus every face of $\Delta[W]$ is contained in some $F_t$, while every
	$F_t$ is itself a face of $\Delta[W]$. Therefore
	\[
	\Delta[W]=\bigcup_{t=1}^{k}\langle F_t\rangle.
	\]
	
	Finally, if $M$ is a facet of $\Delta[W]$, then $M\subseteq F_t$ for
	some $t$. Since $F_t$ is itself a face and $M$ is maximal with respect
	to inclusion, we must have $M=F_t$. Hence the facets of $\Delta[W]$
	are precisely the inclusion-maximal members among
	$F_1,\ldots,F_k$.
\end{proof}

%
%

\begin{proposition}\label{prop:circle}
	Let $n\geq 3m+1$ and let \( 	W=\{i_1,i_2,\ldots,i_k\}\in\mathcal{S}_V(k,m). \)
	Then
	\[
	\Delta[W]\simeq\mathbb{S}^1.
	\]
\end{proposition}

\begin{proof}
	Let $r_1,r_2,\ldots,r_k$ be the successive gaps associated to $W$,
	and let \( F_t=\{i_t,i_{t+1},\ldots,i_{t+s_t}\ \), where \( 1\leq t\leq k \)
	be the faces defined in Lemma~\ref{lem:facetsW}. Recall that
	\[
	s_t=
	\max\left\{
	1\leq s\leq k-1:
	r_t+r_{t+1}+\cdots+r_{t+s-1}\leq m
	\right\}.
	\]
	For each $t$, let \( \Sigma_t=\langle F_t\rangle \)
	denote the simplex having $F_t$ as its vertex set. By
	Lemma~\ref{lem:facetsW}, every face of $\Delta[W]$ is contained in
	some $F_t$. Hence
	\[
	\Delta[W]=\bigcup_{t=1}^{k}\Sigma_t.
	\]
	We first apply the Nerve Lemma to this covering. Since each
	$\Sigma_t$ is a simplex, it is contractible. Moreover, for any
	$t_1,\ldots,t_q$, whenever the intersection is nonempty, we have
	\[
	\Sigma_{t_1}\cap\cdots\cap\Sigma_{t_q}
	=
	\left\langle
	F_{t_1}\cap\cdots\cap F_{t_q}
	\right\rangle.
	\]
	Thus every nonempty finite intersection of the $\Sigma_t$'s is a
	simplex and is therefore contractible. Let \( 	N_W=\mathcal N(\Sigma_1,\ldots,\Sigma_k) \)
	be the nerve of this covering. By the Nerve Lemma, \( \Delta[W]\simeq N_W. \)
	
	It remains to determine the homotopy type of $N_W$. Consider the
	geometric realization of the cycle $\C_n$, which is homeomorphic to
	$\mathbb S^1$. For each $t$, let $A_t$ be the closed positively
	oriented arc of $\C_n$ beginning at $i_t$ and ending at
	$i_{t+s_t}$. We regard each $A_t$ as the corresponding subcomplex
	of the geometric realization of $\C_n$. By the definition of $s_t$,
	the length of $A_t$ is \( 	r_t+r_{t+1}+\cdots+r_{t+s_t-1}\leq m. \)
	Thus every $A_t$ is a proper arc of $\C_n$. Since $n\geq3m+1$,, we
	have \( m<\frac{n}{2}, \)
	so every $A_t$ has length strictly less than one half of the
	circumference of $\C_n$.
	
	We claim that $A_1,A_2,\ldots,A_k$ cover the whole cycle. Indeed,
	since $W\in\mathcal S_V(k,m)$, we have $r_t\leq m$ for every $t$.
	Hence $s_t\geq1$, and therefore $A_t$ contains the entire arc of
	$\C_n$ joining $i_t$ to $i_{t+1}$. The arcs \( [i_1,i_2],\ [i_2,i_3],\ldots,[i_k,i_1] \)
	together traverse $\C_n$ exactly once. Consequently,
	\[
	|\C_n|=A_1\cup A_2\cup\cdots\cup A_k.
	\]
	
	We next show that every nonempty finite intersection of the arcs
	$A_t$ is contractible. Suppose that
\( A_{t_1}\cap\cdots\cap A_{t_q}\neq\varnothing \)
	and choose \( x\in A_{t_1}\cap\cdots\cap A_{t_q}. \)
	Let
	\[
	\pi:\mathbb R\longrightarrow\mathbb R/n\mathbb Z
	\cong |\C_n|
	\]
	be the universal covering map, and fix a lift $\widetilde{x}$ of
	$x$. For each $j$, choose a lift $[a_j,b_j]$ of $A_{t_j}$ containing
	$\widetilde{x}$. After translating simultaneously, we may assume
	that $\widetilde{x}=0$. Thus \( a_j\leq0\leq b_j \) and \( b_j-a_j\leq m. \)
	It follows that
	\[
	\max_j b_j-\min_j a_j\leq2m<n.
	\]
	Hence $\pi$ is injective on \( \bigcup_{j=1}^{q}[a_j,b_j]. \)
	Therefore
	\[
	A_{t_1}\cap\cdots\cap A_{t_q}
	=
	\pi\left(
	\bigcap_{j=1}^{q}[a_j,b_j]
	\right)
	=
	\pi\left(
	\left[\max_j a_j,\min_j b_j\right]
	\right).
	\]
	Since the intersection is nonempty, it is either a closed arc or a
	single point, and hence is contractible. Thus every nonempty finite
	intersection of the subcomplexes $A_1,\ldots,A_k$ is contractible.
	Let \( N_A=\mathcal N(A_1,\ldots,A_k) \)
	be the nerve of this covering. By the Nerve Lemma,
	\[
	N_A\simeq|\C_n|\simeq\mathbb S^1.
	\]
	
	We now compare $N_A$ with $N_W$. We claim that, for any indices
	$t_1,\ldots,t_q$,
	\[
	A_{t_1}\cap\cdots\cap A_{t_q}\neq\varnothing
	\Longleftrightarrow
	F_{t_1}\cap\cdots\cap F_{t_q}\neq\varnothing.
	\]
	If \( i_j\in F_{t_1}\cap\cdots\cap F_{t_q}, \)
	then $i_j$ lies in each of the corresponding arcs, and hence \( 	i_j\in A_{t_1}\cap\cdots\cap A_{t_q}. \)

	Conversely, suppose that \( 	A_{t_1}\cap\cdots\cap A_{t_q}\neq\varnothing. \)
	As shown above, this intersection is either a closed arc or a single
	point. Since the endpoints of an intersection of finitely many
	closed intervals are endpoints of some of the constituent
	intervals, the corresponding intersection of the arcs contains an
	endpoint of one of $A_{t_1},\ldots,A_{t_q}$. Every such endpoint is
	a vertex of $W$. Hence the intersection contains some vertex
	$i_j\in W$. Since the vertices of $W$ lying on $A_{t_\ell}$ are
	precisely the vertices of $F_{t_\ell}$, we obtain \( 	i_j\in F_{t_1}\cap\cdots\cap F_{t_q}. \)
	Therefore \( 	F_{t_1}\cap\cdots\cap F_{t_q}\neq\varnothing. \)
	
	It follows that the two coverings have the same nerve, so \( N_W=N_A. \)
	Consequently,
	\[
	\Delta[W]\simeq N_W=N_A\simeq\mathbb S^1.
	\]
\end{proof}

\begin{lemma}\label{lem:circle}
	Let $n\geq3m+1$ and let \( W\in\mathcal{S}_V(k,m). \)
	Then
	\[
	\widetilde H_r(\Delta[W];\K)
	=
	\begin{cases}
		\K,&r=1,\\
		0,&r\neq1.
	\end{cases}
	\]
\end{lemma}

\begin{proof}
	By Proposition~\ref{prop:circle}, \( \Delta[W]\simeq\mathbb S^1. \)
	Therefore \( 	\widetilde H_r(\Delta[W];\K)
	\cong
	\widetilde H_r(\mathbb S^1;\K), \)
	and hence
	\[
	\widetilde H_r(\Delta[W];\K)
	=
	\begin{cases}
		\K,&r=1,\\
		0,&r\neq1.
	\end{cases}
	\]
\end{proof}

\begin{theorem}\label{thm:betti2}
Let $n\geq3m+1$. Then, for $i\geq1$, \( \beta_{i,i+2}(G_{n,m})
=
\left|\mathcal{S}_V(i+2,m)\right|. \)
Hence,
\[
{
\beta_{i,i+2}(G_{n,m})
=
\frac{n}{i+2}
\sum_{q=0}^{\left\lfloor\frac{n-i-2}{m}\right\rfloor}
(-1)^q
\binom{i+2}{q}
\binom{n-mq-1}{i+1}
}.
\]
In particular, \( \beta_{i,i+2}(G_{n,m})=0 \) if \( i+2<\left\lceil\frac{n}{m}\right\rceil. \)

\end{theorem}

\begin{proof}
By Hochster's formula,
\[
\beta_{i,i+2}(G_{n,m})
=
\sum_{\substack{W\subseteq V\\|W|=i+2}}
\dim_{\K}\widetilde H_1(\Delta[W];\K).
\]
By Lemmas~\ref{lem:circle} and~\ref{lem:contractible}, we see that
\[
\dim_{\K}\widetilde H_1(\Delta[W];\K)
=
\begin{cases}
1,&W\in\mathcal{S}_V(i+2,m),\\
0,&W\notin\mathcal{S}_V(i+2,m).
\end{cases}
\]
Therefore \( \beta_{i,i+2}(G_{n,m})
=
|\mathcal{S}_V(i+2,m)|. \)
The required formula now follows from Lemma~\ref{lem:count}.
\end{proof}

\begin{definition}
	A nonzero graded Betti number $\beta_{i,j}(G)$ is called
	\emph{extremal} if \( \beta_{p,q}(G)=0 \)
	for every $p\geq i$ and $q\geq j+1$ satisfying
	\( q-p\geq j-i. \)
\end{definition}

\begin{theorem}\label{thm:extremal}
Let $n\geq3m+1$. Then \( \beta_{n-2,n}(G_{n,m})=1. \)
Moreover, this is an extremal Betti number.
\end{theorem}

\begin{proof}
Since $V\in\mathcal{S}_V(n,m)$, Lemma~\ref{lem:circle} gives \( \dim_{\K}\widetilde H_1(\Delta;\K)=1. \)
Thus Hochster's formula yields \( \beta_{n-2,n}(G_{n,m})=1. \)
Furthermore, Lemmas~\ref{lem:contractible} and~\ref{lem:circle} show that
\( \widetilde H_r(\Delta[W];\K)=0 \)
for every $W\subseteq V$ and every $r>1$. Hence there are no nonzero Betti
numbers above the second strand, and $\beta_{n-2,n}(G_{n,m})$ is extremal.
\end{proof}

\begin{corollary}\label{cor:regpd}
	Let $n\geq3m+1$. Then \( \reg(G_{n,m})=2 \) and 
	\( \pd(G_{n,m})=n-2. \)
\end{corollary}

\begin{proof}
	By Theorem~\ref{thm:extremal}, \( \beta_{n-2,n}(G_{n,m})=1. \)
	Since \( n-(n-2)=2, \)
	it follows that \( \reg(G_{n,m})\geq2. \)
	On the other hand, Lemmas~\ref{lem:circle} and
	\ref{lem:contractible} show that \( \widetilde H_r(\Delta[W];\K)=0 \)
	for every \( W\subseteq V \) and \( r>1 \).
Hochster's formula therefore implies that \( \beta_{i,j}(G_{n,m})=0 \) whenever \( j-i>2. \)
	Hence \( \reg(G_{n,m})\leq2, \) and therefore
	\[
	\reg(G_{n,m})=2.
	\]
	
	Furthermore, \( \beta_{n-2,n}(G_{n,m})\neq0 \)
	implies \( \pd(G_{n,m})\geq n-2. \)
	The only possible Betti number in a larger homological degree is
	$\beta_{n-1,n}(G_{n,m})$. By Hochster's formula,
	\[
	\beta_{n-1,n}(G_{n,m})
	=
	\dim_{\K}\widetilde H_0(\Delta;\K).
	\]
	Since Proposition~\ref{prop:circle} gives \( \Delta\simeq\mathbb{S}^{1}, \) the complex $\Delta$ is connected, and therefore \( \widetilde H_0(\Delta;\K)=0. \)
	Thus \( \beta_{n-1,n}(G_{n,m})=0, \)
	and hence
	\[
	\pd(G_{n,m})=n-2.
	\]
\end{proof}
\begin{corollary}\label{cor:im}
	Let $n\geq3m+1$. Then
	\[
	\operatorname{im}(G_{n,m})
	=
	\begin{cases}
		2, & 3m+1\leq n\leq4m,\\[2mm]
		1, & n\geq4m+1.
	\end{cases}
	\]
\end{corollary}

\begin{proof}
	Suppose first that $3m+1\leq n\leq4m$. By Remark~\ref{rem:im},
	\( \operatorname{im}(G_{n,m})\geq2. \)
	On the other hand, Katzman's bound gives \( 	\operatorname{im}(G_{n,m})
	\leq
	\reg(G_{n,m}). \)
	By Corollary~\ref{cor:regpd}, \( \reg(G_{n,m})=2. \)
	Hence \( \operatorname{im}(G_{n,m})=2. \)
	If $n\geq4m+1$, then Proposition~\ref{prop:im} gives \( \operatorname{im}(G_{n,m})=1. \)
\end{proof}

\section{Hilbert series and the linear strand}
\noindent
Let $\Gamma$ be a simplicial complex of dimension $d-1$. For
$-1\leq q\leq d-1$, let $f_q(\Gamma)$ denote the number of
$q$-dimensional faces of $\Gamma$, where $f_{-1}(\Gamma)=1$. The vector
\[
f(\Gamma)
=
(f_{-1},f_0,\ldots,f_{d-1})
\]
is called the $f$-vector of $\Gamma$.
On the other hand, the $h$-vector
\[
h(\Gamma)=(h_0,h_1,\ldots,h_d)
\]
is defined by the identity
\[
\sum_{j=0}^{d}h_jt^j
=
\sum_{j=0}^{d}
f_{j-1}t^j(1-t)^{d-j}.
\]
Equivalently,
\[
h_j
=
\sum_{r=0}^{j}
(-1)^{j-r}
\binom{d-r}{j-r}f_{r-1}.
\]

For a finitely generated graded $R$-module \( M=\bigoplus_{q\geq0}M_q, \) its Hilbert series is
\[
H(M;t)
=
\sum_{q\geq0}
\dim_{\K}(M_q)t^q.
\]
For a $(d-1)$-dimensional simplicial complex $\Gamma$, the Hilbert series
of its Stanley--Reisner ring is
\[
H(\K[\Gamma];t)
=
\frac{h_0+h_1t+\cdots+h_dt^d}{(1-t)^d}.
\]

We now compute the $f$-vector and the $h$-vector of $\Delta$.

\begin{lemma}\label{lem:fvector}
Let $n\geq3m+1$, and let $f_q=f_q(\Delta)$ denote the number of
$q$-dimensional faces of $\Delta$. Then \( f_0=n \)
and, for $1\leq q\leq m$, \( f_q=n\binom{m}{q}. \)
Moreover, \( f_q=0 \) for \( q>m. \)
Thus
\[
f(\Delta)
=
\left(
1,n,n\binom{m}{1},n\binom{m}{2},\ldots,n\binom{m}{m}
\right).
\]
\end{lemma}
\begin{proof}
	Since \( \Delta=\operatorname{Cl}(\C_n^{[m]}), \)
	the faces of $\Delta$ are precisely the cliques of $\C_n^{[m]}$.
	Clearly, \( f_0=n. \)
	
	Let $q\geq1$ and let $F$ be a $(q+1)$-element clique. By
	Lemma~\ref{lem:arc}, the vertices of $F$ are contained in an arc of
	$\C_n$ of length at most $m$. Since
	\( m<\frac n2, \)
	the shortest arc containing $F$ is unique. Let $v$ denote its first
	vertex in the positive direction. Necessarily,
	\( v\in F. \)
	
	Once $v$ is fixed, every other vertex of $F$ must belong to the next
	$m$ vertices \( v+1,v+2,\ldots,v+m. \)
	
	Conversely, if $v$ is fixed and $q$ vertices are chosen from these
	$m$ vertices, then the resulting $(q+1)$-element set lies in the arc
	from $v$ to $v+m$. Hence every two of its vertices are at distance at
	most $m$ in $\C_n$, and therefore it is a clique of $\C_n^{[m]}$.
	Moreover, since $m<n/2$, the vertex $v$ is the unique first vertex of
	its shortest containing arc.
	
	Thus the $(q+1)$-element cliques are in bijection with pairs consisting
	of a choice of the first vertex $v$ and a choice of $q$ vertices among
	the next $m$ vertices. Therefore \[	f_q=n\binom{m}{q},\]
	where \( 1\leq q\leq m. \)
	No clique can contain more than $m+1$ vertices, while \( \{v,v+1,\ldots,v+m\} \)
	is a clique of cardinality $m+1$. Hence \( \dim\Delta=m \)
	and \( f_q=0 \) for \( q>m. \)
\end{proof}

\begin{lemma}\label{lem:hvector}
Let \( h(\Delta)=(h_0,h_1,\ldots,h_{m+1}) \)
be the $h$-vector of $\Delta$. Then \( h_0=1, h_1=n-m-1, \) 
and for  \( 2\leq j\leq m+1 \),
\[
h_j=(-1)^j\binom{m+1}{j}.
\]
\end{lemma}

\begin{proof}
	Since $\dim\Delta=m$, put \( d=m+1. \)
	By the defining relation between the $f$-vector and the $h$-vector,
	\[
	\sum_{j=0}^{m+1}h_jt^j
	=
	\sum_{r=0}^{m+1}
	f_{r-1}t^r(1-t)^{m+1-r}.
	\]
	Using \( f_{-1}=1 \) and, by Lemma~\ref{lem:fvector}, \( f_{r-1}=n\binom{m}{r-1}, \) \( 1\leq r\leq m+1, \)
	we obtain
	\[
	\begin{aligned}
		\sum_{j=0}^{m+1}h_jt^j
		&=
		(1-t)^{m+1}
		+
		n\sum_{r=1}^{m+1}
		\binom{m}{r-1}t^r(1-t)^{m+1-r}\\
		&=
		(1-t)^{m+1}
		+
		nt\sum_{q=0}^{m}
		\binom{m}{q}t^q(1-t)^{m-q}.
	\end{aligned}
	\]
	By the binomial theorem, \( 	\sum_{q=0}^{m}
	\binom{m}{q}t^q(1-t)^{m-q}
	=
	\bigl(t+(1-t)\bigr)^m
	=
	1. \)
	Hence
	\[
	\sum_{j=0}^{m+1}h_jt^j
	=
	(1-t)^{m+1}+nt.
	\]
	Expanding $(1-t)^{m+1}$ gives
\( h_0=1, h_1=n-m-1, \)
	and for \( 2\leq j\leq m+1 \),
	\[
	h_j=(-1)^j\binom{m+1}{j}.
	\]
\end{proof}

\begin{remark}
For $m=2$, Lemma~\ref{lem:hvector} gives \( h(\Delta)=(1,n-3,3,-1), \)
which agrees with the corresponding result for complements of squares
of cycles~\cite{RatherSquare}.
\end{remark}

Since $\Delta$ has $n$ vertices and $\dim\Delta=m$,
its Stanley--Reisner ring has Hilbert series
\[
H(\K[\Delta];t)
=
\frac{h_0+h_1t+\cdots+h_{m+1}t^{m+1}}
{(1-t)^{m+1}}.
\]
In order to express this Hilbert series with denominator $(1-t)^n$, put
\[
\widehat h(t)
=
(1-t)^{n-m-1}
\big(h_0+h_1t+\cdots+h_{m+1}t^{m+1}\big).
\]
Writing \( \widehat h(t)
=
\sum_{r=0}^{n}\widehat h_rt^r, \)
we obtain
\[
\widehat h_r
=
\sum_{j=0}^{r}
(-1)^{r-j}
\binom{n-m-1}{r-j}h_j,
\]
where $h_j=0$ for $j>m+1$. Therefore
\[
H(\K[\Delta];t)
=
\frac{\sum_{r=0}^{n}\widehat h_rt^r}{(1-t)^n}.
\]

\begin{theorem}\label{thm:linear}
	Let $n\geq3m+1$. For every $i\geq0$, \( \beta_{i+1,i+2}(G_{n,m})
	=
	(-1)^{i+1}\widehat h_{i+2}
	+
	\beta_{i,i+2}(G_{n,m}), \)
	where
	\[
	\widehat h_r
	=
	\sum_{j=0}^{r}
	(-1)^{r-j}
	\binom{n-m-1}{r-j}h_j
	\]
	and \( h(\Delta)=(h_0,h_1,\ldots,h_{m+1}) \)
	is as given in Lemma \ref{lem:hvector}.
\end{theorem}

\begin{proof}
By Corollary~\ref{cor:regpd}, the minimal graded free resolution of the
Stanley--Reisner ring $\K[\Delta]$ has nonzero Betti numbers only in the
strands $\beta_{i,i+1}$ and $\beta_{i,i+2}$. Hence
\[
\mathcal{H}(\K[\Delta];t)
=
\frac{
1+\sum_{i=1}^{n-2}(-1)^i
\big(\beta_{i,i+1}t^{i+1}+\beta_{i,i+2}t^{i+2}\big)
}{(1-t)^n}.
\]
On the other hand,
\[
\mathcal{H}(\K[\Delta];t)
=
\frac{\sum_{r=0}^{n}\widehat h_rt^r}{(1-t)^n}.
\]
Comparing the coefficients of $t^{i+2}$ gives
\[
\widehat h_{i+2}
=
(-1)^{i+1}\beta_{i+1,i+2}
+
(-1)^i\beta_{i,i+2}.
\]
Therefore \( \beta_{i+1,i+2}
=
(-1)^{i+1}\widehat h_{i+2}
+
\beta_{i,i+2}. \)
Substituting the formulas for $\widehat h_{i+2}$ and
$\beta_{i,i+2}$ proves the result.
\end{proof}
\begin{corollary}
	For $i\geq0$,
	\[
	\begin{aligned}
		\beta_{i+1,i+2}(G_{n,m})
		={}&
		(-1)^{i+1}
		\sum_{j=0}^{i+2}
		(-1)^{i+2-j}
		\binom{n-m-1}{i+2-j}h_j\\
		&+
		\frac{n}{i+2}
		\sum_{q=0}^{\left\lfloor\frac{n-i-2}{m}\right\rfloor}
		(-1)^q
		\binom{i+2}{q}
		\binom{n-mq-1}{i+1},
	\end{aligned}
	\]
	where $h_j$ is as in Lemma~\ref{lem:hvector}.

\end{corollary}

\begin{example}\label{ex:C123}
	Consider the complement of the third closed power of the $12$-cycle, \( G_{12,3}=\overline{\C_{12}^{[3]}}. \)
	Since $12\geq3(3)+1$, the results obtained above apply to this graph.
	We compute all the graded Betti numbers of the edge ring
	$R/I(G_{12,3})$, where
	$R=\K[x_1,\ldots,x_{12}]$.
	
	We first determine the Betti numbers in the second strand. By
	Theorem~\ref{thm:betti2},
	\[
	\beta_{i,i+2}(G_{12,3})
	=
	\frac{12}{i+2}
	\sum_{q=0}^{\left\lfloor\frac{10-i}{3}\right\rfloor}
	(-1)^q
	\binom{i+2}{q}
	\binom{11-3q}{i+1}.
	\]
	Since \( \left\lceil\frac{12}{3}\right\rceil=4, \)
	we have $\beta_{1,3}(G_{12,3})=0$. For the remaining values of $i$,
	the above formula gives
	\[
	\begin{array}{c|ccccccccc}
		i&2&3&4&5&6&7&8&9&10\\
		\hline
		\beta_{i,i+2}
		&3&72&282&456&399&208&66&12&1.
	\end{array}
	\]
	
	Next we compute the linear strand. By Lemma~\ref{lem:fvector}, the
	$f$-vector of $\Delta=\Delta(G_{12,3})$ is
	\[
	f(\Delta)=(1,12,36,36,12).
	\]
	By Lemma~\ref{lem:hvector},
	\[
	h(\Delta)=(1,8,6,-4,1).
	\]
	Since $n-m-1=8$, we have
	\[
	\widehat h(t)
	=
	(1-t)^8(1+8t+6t^2-4t^3+t^4).
	\]
	Expanding, we obtain
	\[
	\widehat h(t)
	=
	1-30t^2+116t^3-177t^4+48t^5+252t^6
	-456t^7+399t^8-208t^9
	+66t^{10}-12t^{11}+t^{12}.
	\]
	Therefore, Theorem~\ref{thm:linear} gives
	\[
	\begin{aligned}
		\beta_{1,2}
		&=-\widehat h_2=30,\\
		\beta_{2,3}
		&=\widehat h_3+\beta_{1,3}=116,\\
		\beta_{3,4}
		&=-\widehat h_4+\beta_{2,4}=177+3=180,\\
		\beta_{4,5}
		&=\widehat h_5+\beta_{3,5}=48+72=120,\\
		\beta_{5,6}
		&=-\widehat h_6+\beta_{4,6}=-252+282=30.
	\end{aligned}
	\]
	For the remaining indices, the two terms in
	Theorem~\ref{thm:linear} cancel. Thus
	\[
	\beta_{6,7}
	=
	\beta_{7,8}
	=
	\beta_{8,9}
	=
	\beta_{9,10}
	=
	\beta_{10,11}
	=
	\beta_{11,12}
	=
	0.
	\]

	Hence the minimal graded free resolution has the form
\[
\begin{aligned}
	0\longrightarrow&
	R(-12)
	\longrightarrow R(-11)^{12}
	\longrightarrow R(-10)^{66}
	\longrightarrow R(-9)^{208}
	\longrightarrow R(-8)^{399}
	\longrightarrow
	\substack{
		R(-7)^{456}\\
		\oplus\\
		R(-6)^{30}
	}\\
	\longrightarrow	&
	\substack{
		R(-6)^{282}\\
		\oplus\\
		R(-5)^{120}
	}
	\longrightarrow
	\substack{
		R(-5)^{72}\\
		\oplus\\
		R(-4)^{180}
	}
	\longrightarrow
	\substack{
		R(-4)^3\\
		\oplus\\
		R(-3)^{116}
	}
	\longrightarrow
	R(-2)^{30}
	\longrightarrow R
	\longrightarrow R/I(G_{12,3})
	\longrightarrow0.
\end{aligned}
\]
The corresponding Betti table, computed using Macaulay-2~\cite{M2}, is
given below and agrees exactly with the Betti numbers obtained from our
formulas.
\[
\begin{array}{c|rrrrrrrrrrr}
	&0&1&2&3&4&5&6&7&8&9&10\\
	\hline
	\text{total:}
	&1&30&119&252&402&486&399&208&66&12&1\\
	0:
	&1&.&.&.&.&.&.&.&.&.&.\\
	1:
	&.&30&116&180&120&30&.&.&.&.&.\\
	2:
	&.&.&3&72&282&456&399&208&66&12&1
\end{array}
\]
Here the entry in column $i$ and row $j-i$ is $\beta_{i,j}$. In
particular, $\reg(G_{12,3})=2$ and $\pd(G_{12,3})=10$, in agreement
with Corollary~\ref{cor:regpd}.
\end{example}

%
%
\subsection*{Conflict of Interest}

The author declares that there is no conflict of interest.

\subsection*{Data Availability Statement}

All data generated or analysed during this study are included in this
article.

\end{document}